\documentclass{amsproc}
\usepackage{epsfig,latexsym,amsfonts,amssymb,amsmath,amscd}

\newtheorem{theorem}{Theorem}[section]
\newtheorem{lemma}[theorem]{Lemma}
\newtheorem{lem}[theorem]{Lemma}

\theoremstyle{definition}
\newtheorem{definition}[theorem]{Definition}
\newtheorem{example}[theorem]{Example}

\newtheorem{prop}[theorem]{Proposition}

\newtheorem{defn}[theorem]{Definition}
\theoremstyle{remark}
\newtheorem{remark}[theorem]{Remark}
\newtheorem{rem}[theorem]{Remark}
\numberwithin{equation}{section}

\def\mcA{\mathcal A}
\def\mcF{\mathcal F}
\def\val{\operatorname{val}}

\def\mcH{\mathcal H}
\def\mcJ{\mathcal J}
\def\mcS{\mathcal S}

\def\mcM{\mathcal M}
\def\mcN{\mathcal N}
\def\mcS{\mathcal S}
\def\C{\mathbb C}
\def\R{\mathbb R}
\def\nsets{\mathcal P^\sharp}
\def\Cong{{\operatorname{Cong}}}

\newcommand{\etype}[1]{\renewcommand{\labelenumi}{(#1{enumi})}}
\def\eroman{\etype{\roman}}

\newcommand{\Net}{\mathbb N}

\newcommand{\trop}[1]{\mathcal{#1}}

\newcommand{\tG}{\trop{G}}

\newcommand{\one}{1}
\newcommand{\zero}{0}

\newcommand{\tT}{\trop{T}}
\def\tTz{\trop{T}_0}

\def\cocoa{{\hbox{\rm C\kern-.13em o\kern-.07em C\kern-.13em o\kern-.15em A}}}

\def\w2M{\bigwedge^2M}

\def\w{\wedge }

\protect
\protect
\protect \protect

\protect
\protect

\def\be{\begin{equation}}
\def\ee{\end{equation}}
\def\bclm{\begin{claim}}
\def\eclm{\end{claim}}
\def\beqn{\begin{eqnarray}}
\def\eeqn{\end{eqnarray}}
\def\beqn*{\begin{eqnarray*}}
\def\eeqn*{\end{eqnarray*}}

\numberwithin{equation}{section}
\begin{document}

\title{Residue Structures}

\author{Louis H. Rowen}

\address{Department of Mathematics, Bar-Ilan University, Ramat-Gan 52900,
Israel}
\curraddr{Department of Mathematics, Bar-Ilan University, Ramat-Gan 52900,
Israel}
\email{rowen@math.biu.ac.il}
 \thanks{The research of the author was supported by the ISF grant 1994/20 and the Anshel Peffer Chair.}


\date{February 2024}


\begin{abstract}    We consider residue structures $R/G$ where $(G,+)$ is  an additive subgroup of a   ring $(R,+,\cdot)$, not necessarily an ideal.   Special instances include  Krasner's construction of quotient hyperfields, and Pumpluen's  construction of nonassociative algebras. The residue construction, treated formally, satisfies the Noether isomorphism theorems, and also is cast in a broader categorical setting   which includes categorical  products, sums, and tensor products.
\end{abstract}
\maketitle

\section{Introduction and preliminaries}

 The residue group $\mcA/G$ of a group $\mcA$ modulo a normal subgroup $G$ is one
of the most familiar setups in algebra, lying at the foundation of group theory. The cosets form the residue group,
as learned in a beginning  course  in abstract algebra.

Analogously, given a ring $R$ and a subgroup $G$ of $(R,+)$, one
might ask about the structure of $R/G$. In order for the cosets to be a ring,
$G$ must be an ideal of $R$.
In this paper we consider what happens when we drop this additional assumption on $G$. Krasner \cite{krasner1,krasner} achieved considerable success in field theory while using this approach, and it also can be applied to  recover a nonassociative algebra of Dickson described in \cite{Pu}.

The Krasner construction takes us to the power set $\mathcal{P}(R)$ of a set $R,$ in which some but not all of the algebraic properties of $R$ are preserved.
  In particular, one loses additive inverses, and is drawn from ring theory to semiring theory, and beyond. Our  purpose is to lay categorical  foundations for this residue construction, providing the relevant functors tying the original ring with the Krasner residue.

  We cast the theory for magmas, in order to include more varied algebraic constructions.

\subsection{Underlying algebraic structures}$ $

First we review some definitions.
$\Net^*$ denotes the positive natural numbers, and   $\Net = \Net^* \cup{\zero}.$

\begin{definition}\label{mag1} $ $ \eroman
\begin{enumerate}
    \item
 A \textbf{magma} $(\mcM,*,\iota)$ is a set $\mcM$ with a binary operation $ \mcM   \times  \mcM    \to \mcM$, not necessarily associative,  where the  operation is denoted by concatenation.  In this paper a magma always has a neutral element, denoted $\iota$, i.e., $\iota * b =b *\iota = b$ for all $b\in \mcM$. \footnote{  In \cite{NakR} these magmas are called {\it unital}.}

\item A \textbf{monoid} is   a magma with an
associative operation, not necessarily commutative.  If its operation is denoted by ``$+$" (resp.~``$\cdot$'') then the neutral element is denoted by $\zero$ (resp. $\one$).

 \item   A submonoid $\mcN$ of  a  monoid  $(\mcM,\cdot)$ is \textbf{normal} if $\mcN \cdot a=a\cdot\mcN $ for all $a\in \mcM.$ (This generalizes the notion of a normal subgroup of a group.)
\end{enumerate}
 \end{definition}

We often delete $\cdot$ from the notation, using concatenation.
Our point of departure is the following basic structure result.
\begin{lem}\label{resmon}$ $
    \begin{enumerate}\eroman
         \item If $G$ is a normal submonoid of a monoid $\tT$, then there is the \textbf{residue monoid}  $\bar \tT := \tT/G$, the set of cosets, which has the operation $$\bar a_1\bar a_2 = \overline{a_1a_2},$$ writing $\bar a_i$ for $a_iG.$
          \item $\bar \tT$ is a group if and only if for every $a\in \tT$ there is $a'\in \tT$ such that $aa' \in G.$
    \end{enumerate}
\end{lem}
\begin{proof}(i)
    Just as for the familiar case of groups. Namely $(a_1G)(a_2G)= (a_1Ga_2)G = (a_1a_2)GG =(a_1a_2)G .$ The neutral element is $\bar\iota= \iota G.$

    (ii) $aa' \in G$ if and only if $\bar a \bar a' =\overline{aa'} = \bar\iota.$
\end{proof}

\subsection{Semirings and semifields}

Our interest lies in ring-like structures.
\begin{definition}\label{mag11} $ $ \eroman
\begin{enumerate}
     \item
 A \textbf{pre-semiring} $(\mcS,\cdot,+,\one)$ (called a bimagma in \cite{JMR} and \cite{GaR2}) is both a multiplicative
 monoid $(\mcS,\cdot,\one)$ and a commutative associative additive  monoid $(\mcS,+,\zero).$
  We can always adjoin a $\zero$ element that is additively neutral and also is   multiplicatively absorbing,  so we assume that any pre-semiring also has such a $\zero$.  We shall denote multiplication by concatenation. A pre-semiring  is \textbf{associative} if multiplication is associative.

 \item
 A \textbf{semiring} \cite{Cos,Golan} is a pre-semiring $\mcS$  that satisfies all the properties of a ring (including distributivity of multiplication over addition), except possibly
negation.

\item
A semiring $\mcS$ is a \textbf{semifield} if $ \mcS^\times:= \mcS \setminus \{0\}$ is a group (not necessarily abelian).
\end{enumerate}
 \end{definition}

\begin{example} Here is the best known  example of a semifield which is not a ring, which plays a major role in tropical algebra.
    If $(\tG,+)$ is an ordered abelian  monoid, the \textbf{max-plus algebra} on $\mcA :=\tG \cup \{-\infty\}$ is given  by defining multiplication to be the original addition on $\tG,$ with $-\infty$ additively absorbing, and addition to be the maximum. Thus the   zero element of $\mcA$ is
    $-\infty$.
\end{example}

\subsubsection{The structure of semifields}$ $

Semirings (as well as magmas  and monoids) readily fit into the framework of universal algebra \cite[Chapter 2]{Jac80}, in which homomorphisms are defined via congruences. Although the theory is quite elegant, one finds it  difficult to deal explicitly with congruences, so the roles of addition and multiplication in semifields were reversed in \cite{HuW,VeC,WeW} in order to utilize group theory.

\begin{definition}\label{ker}
     A \textbf{kernel} of a semifield $\mcS$ is a normal subgroup  $K$ of
  $\mcS^\times$ which is \textbf{convex} in the sense that if  $a_1,a_2
\in {K}$ and $r_1, r_2 \in \mcS$ with $r_1 + r_2 =\one,$ then $r_1 a_1
+ r_2 a_2 \in {K}.$
\end{definition}

(Of course one may assume that $a_1=\one.$)
We have the following key correspondence, given in  \cite[Theorem
~3.2]{HuW}:

\begin{prop}\label{basicprop} If $\Cong$ is a congruence
on a semifield $F$, then ${K}_\Cong = \left\{ a  \in F : a \equiv
\one\right\}$ is a kernel. Conversely, any kernel $K$ of $F$
defines a congruence  by $a_1 \equiv a_2$ iff $K{a_1}=K{a_2}.$

The semifield structures of $R/{K}_\Cong $ and $R/K$ (under the multiplication of Lemma~\ref{resmon} and addition $[a_1] + [a_2] = [a_1+a_2])$ are the same, thus providing us an isomorphism of categories.\footnote{In the broader context of  semirings one could  define a \textbf{kernel} of a semiring $\mcS$ to be a normal multiplicative submonoid
  ${K}$ which is \textbf{convex} in the sense that if  $a_1,a_2
\in {K}$ and $r_1, r_2 \in \mcS$ with $r_1+r_2 = \one$  then $r_1 a_1
+ r_2 a_2 \in {K}(a_1+a_2).$ Then the proof of \cite[Theorem
~3.2]{HuW}, mutatis mutandis, shows that the kernels of semirings    correspond to the congruences. We leave the proof for the reader, since it does not affect our subsequent results.}
\end{prop}

Other equivalent criteria are given in \cite[Theorem~3.2]{HuW}.
Thus, the category of semifields in universal algebra, is isomorphic to the category of lattice-ordered abelian groups, where the kernels correspond to the convex subgroups, cf.~\cite{WeW}. Tropical applications are found in   \cite{PR}. We are lead to ask, ``What happens when one drops the convexity condition on a kernel?''  In order to frame this question more broadly, we introduce a rather inclusive algebraic notion.

    \section{$\tT$-bimagmas and $\tT$-bimodules}

    The notion of an algebraic structure having a designated subgroup  was considered formally in \cite{JMR}, and developed in  \cite{AGR2}.
 We take the essence from \cite{AGR2}, slightly more generally, taking \cite{NakR} into account.

  \begin{definition}\label{Tmagm}
Let $\tT$ be a  monoid with a designated element $\one$.
\eroman
\begin{enumerate}

  \item A   \textbf{left $\tT$-set} is a set $\mcA$  together with a  (left)   $\tT$-action $\tT\times \mathcal A \to \mathcal A$ (denoted  as
concatenation), for which   \begin{enumerate}

    \item $\one b = b,$ for all $b\in \mcA$.

 \item    $(a_1a_2)b = a_1(a_2b)$ for all $a_i\in \tT,$ $ b\in \mcA$.

\end{enumerate}

  \item  A  $\tT$-\textbf{biset} is a left and right $\tT$-{set} $\mathcal A$, for which  $(a_1b) a_2 = a_1(b a_2)$ for all $a_i\in \tT$ and $b\in \mcA.$

  \item
A $\tT$-\textbf{bimagma}
 is a magma
$(\mathcal A,*,\iota)$, which also is a  $\tT$-biset satisfying
\begin{enumerate}
 \item $\iota$ is  absorbing, i.e. $a \iota  = \iota   = \iota a, \  \text{for all}\; a \in \tT.$

  \item $\iota$ is the neutral element of $\mcA$.\footnote{The element $\iota$ is called  $e$ in  \cite{NakR}, but we use $e$ for another purpose, cf.~\S\ref{propN1}.  But if a magma~$\mcA$ did not already contain a neutral element $\iota,$ we already could adjoin it formally by declaring  the operations  $\iota * b =b *\iota = b$ for $b\in \mcA$, and $a\iota =\iota a =\iota$ for  $a\in \tT.$}

  \item The action is \textbf{distributive} over $\tT$,  in the sense that
$$a(b_1*b_2) = ab_1 *ab_2,\quad \text{for all}\; a \in \tT,\; b_i \in \mathcal A.$$
\end{enumerate}

When $\tT$ is to be specified, we call $\mcA$ a \textbf{$\tT$-bimagma}.

  \item  A  $\tT$-bimagma   $(\mathcal A,*)$ is \textbf{left normalizing} if $(a _1* \mcA)*(a _2* \mcA)  = (a _1 a_2* \mcA) $ for all $a_i\in \tT.$

    \item  A  $\tT$-bimagma   $(\mathcal A,*)$ is \textbf{right normalizing} if $( \mcA *a _1)*( \mcA*a _2)  = ( \mcA* a _1 a_2) $ for all $a_i\in \tT.$
 \end{enumerate}
\end{definition}
We make an important restriction.
\begin{definition}$ $\begin{enumerate}\eroman
 \item A   $\tT$-bimagma $\mcA$  is   \textbf{weakly admissible} if $\tT \subseteq \mcA$.   We define $\tTz =\tT\cup \{\iota\},$ and declare $\iota \mcA = \mcA\iota = \iota.$ This makes $\tTz$ a monoid, and $\mcA$ a $\tTz$-bimagma.

    \end{enumerate}
    \end{definition}

A formalism to make a $\tT$-bimagma weakly admissible is given in \cite[Lemma~2.5]{JuR2}.

\medskip

   \begin{definition}$ $\begin{enumerate}\eroman
       \item  We say that a $\tT$-{bimagma} $\mathcal A$ is a $\tT$-\textbf{bimodule} (called {\it mosaic} in \cite{NakR}), and write $+$ in place of~$*$, when   $(*)$ is associative and abelian.
       (Note in this case that the actions $\tT \times \mcA \to \mcA$ and $\mcA \times  \tT\to \mcA$ are linear over $\mcA$)

  \item A $\tT$-\textbf{pre-semiring} is a pre-semiring  which is a unital, weakly admissible $\tT$-{bimodule}   such that $a \cdot b = ab$ and $b\cdot a = ba$ for all $a \in \tT$ and $b\in\mcA.$

    \item A $\tT$~\textbf{semiring}      is a $\tT$-pre-semiring   which is a semiring.

   \end{enumerate}
   \end{definition}

  These definitions are formulated rather generally, in order to include various applications.
 \begin{example}
     \label{Impex}$ $
     \begin{enumerate}\eroman
         \item $\mcA$ is an associative ring  with a multiplicative subgroup $G=\tT$ (our main example). Special case: where $\mcA$ is the group algebra $F[G]$.

    \item $\mcA$ is an associative ring with a left ideal $L=\tT$.
          \item $\mcA$ is a  nonassociative ring with an additive subgroup $\tT$.
         \item  $\mcA$ is a monoid with unit element $\iota$,  and $\tT$ is its set of monoid endomorphisms.
         \item  $\mcA$ is a field,  and $\tT$ is its group of  automorphisms.
     \end{enumerate}
 \end{example}

\section{Hypermagmas}$ $

To handle arbitrary subgroups in structure theory,
 we need a notion originally studied by Marty~\cite{Mar}. We follow the general treatment of Nakamura and Reyes~\cite{NakR}. Define $\nsets(\mcH) := \mathcal{P}(\mcH) \setminus \emptyset.$
\begin{defn}[\cite{Mar}, \cite{NakR}]\label{Hyp00}  $ $
    \begin{enumerate}\eroman
    \item
A  \textbf{hypermagma} $(\mathcal{H},*,\iota)$ is a set with  a binary operation $$*:\mcH \times \mcH \to \nsets(\mcH),\footnote{\cite{NakR} permits  the product of two elements to be the empty set  $\emptyset$,  called $\infty$ in \cite{NakR}, with $\emptyset * S = S * \emptyset = \emptyset $.}$$ with neutral element $\iota$ (i.e., $\iota * a = a * \iota = \{a\})$, whose natural  extension to  $ \nsets(\mcH)$  obtained by defining $S_1 * S _2:= \cup _{s_i \in S_i}  s_1 * s_2$  for     $S_1,S_2\in \nsets(\mathcal{H})$,    makes $ \nsets(\mcH)$ a weakly admissible left and right $\mcH$-bimagma,   viewing $\mcH \subseteq \nsets(\mcH)$ by identifying $a$ with~$\{ a\}$.

  \item A \textbf{sub-hypermagma} $\mcH'$ of $\mcH$  satisfies $a_1* a_2 \subseteq \mathcal{P}(\mcH')$ for $a_1, a_2\in \mcH'$.

   \item
 A \textbf{hyperinverse} of an element $a$ in a hypermagma $(\mathcal{H},*,\iota)$ (if it exists) is an
element denoted ``$a^{-1}$'' for which $\iota \in a * a^{-1}$ and $\iota \in a^{-1} * a$.

 \item
 A   hypermagma $\mcH$ is
\textbf{reversible} when it satisfies the condition
$ a_3 \in a_1 \boxplus a_2$    iff  $ a_2 \in a_1^{-1} * a_3
 $ and  $ a_1 \in a_3
* a_2^{-1} $.\footnote{This is impossible unless $\mcH$ has hyperinverses.}

   \item
A  $\tT$-\textbf{hypermagma} $(\mathcal{H},*)$ is a hypermagma which also is a $\tT$-set with the $\tT$-action distributing over (*), i.e., $a(a_1*a_2) = (aa_1)*(aa_2) $ and $(a_1*a_2)a = (a_1 a)*(a_2 a) $ for $a\in \tT,$ $a_i\in \mcH$.

\item  We often follow the customary notation of $\boxplus$ instead of $*$.   A \textbf{hypermonoid}\footnote{which perhaps should   be called a ``hyperadditive-monoid'' in view of the next definition}  is a hypermagma
$(\mathcal{H},\boxplus,\zero)$, where
the operation
$\boxplus$, called \textbf{hyperaddition}, is  associative (but not necessarily abelian) in the sense that   $(a_1
\boxplus a_2) \boxplus a_3 = a_1 \boxplus (a_2\boxplus a_3)$ for all
$a_i$ in $\mathcal{H}.$
  Then  \begin{enumerate}\eroman
  \item
$\zero \in \mathcal{H}$ is the     \textbf{hyperzero}  satisfying   $\zero \boxplus a = a = a \boxplus \zero,$ for every $a\in \mathcal{H}.$

 \item
 A {hyperinverse} of an element $a$  is called a \textbf{hypernegative} and denoted as  ``$-a$.''
 \end{enumerate}

  \item
A \textbf{hypergroup} is a   hypermonoid $(\mathcal{H},\boxplus,\zero)$ for which every element~$a$
has a unique {hypernegative}, and which is
\textbf{reversible} in the following sense:

$ a_3 \in a_1 \boxplus a_2$    iff  $ a_2 \in a_3
\boxplus (-a_1)$.

 \item
A \textbf{hypersemiring} (resp.~\textbf{hyperring}) is an additive \footnote{In \cite{NakR} this is called ``canonical.'' Other authors such as \cite{KuLS} use the term ``canonical'' for what we call ``reversible.''} hypermonoid (resp. hypergroup) $\mathcal H $, which is also a monoid, $(\mcH,\cdot)$, with $\cdot$ distributing over hyperaddition.

 \item
A \textbf{hyperfield} is a hyperring $\mathcal H $ for which $(\mcH \setminus \{ \zero\},\cdot)$ is a group.

  \item A $\tT$-\textbf{hypermonoid} $(\mathcal{H},\boxplus,\zero)$ satisfies
$a(\boxplus a_i) = \boxplus (aa_i)$ and $(\boxplus a_i)a = \boxplus (a_ia) $ for every $a,a_i \in \mathcal{H}.$

\item
A $\tT$-\textbf{hypersemiring} (resp.~$\tT$-\textbf{hyperring}, resp. ~$\tT$-\textbf{hyperfield}) is a $\tT$-hypermonoid $\mathcal H $ which is a hypersemiring  (resp.~{hyperring},  {hyperfield}) satisfying $a a_1 = a\cdot a_1$ for all $a\in \tT,$ $a_1\in \mcH$.
  \end{enumerate}
 \end{defn}
 \begin{rem}\label{hp1} $ $
 \begin{enumerate}\eroman


   \item If one works with all of $\mathcal{P}(\mcH)$,  $\emptyset$ is the $\infty$ element in the sense of the footnote to Definition~\ref{mag1}.

 \item  Any hypersemiring $\mcH$ is an $\mcH$-hypersemiring where we define $aa_1 = a\cdot a_1. $ In particular, any hyperfield $\mcF$ is an $\mcF$-hyperfield.

   \item One can easily present the noncommutative version, with $(*)$ as the hyperoperation; then the {hyperinverses} are required  to satisfy  $ (a_1*a_2)^{-1} = (a_2^{-1} ) * (a_1^{-1} ).$ But  mostly we follow the more standard $\boxplus$ notation.
\end{enumerate}   \end{rem}

   \subsection{Examples of hyperfields}$ $

   Here are some major examples of hyperfields, recalled from \cite{AGR1,CC,krasner,Vir}.
A~large collection  is given in \cite[\S 2]{MasM}. An assortment of other related examples is found in \cite{NakR}.

\begin{example}\label{ex-krasner}$ $
\begin{enumerate}\eroman
    \item   The \textbf{Krasner hyperfield} \cite{krasner}   is the set $ \mathcal K := \{ 0, 1 \}$, with the usual multiplication
law, and with hyperaddition defined by  $ x \boxplus  0 = 0 \boxplus  x =
x $ for all $x ,$ and $1 \boxplus  1 =  \{ 0, 1\}$.
\item
  The \textbf{hyperfield of signs}
$ \mathcal S := \{ 1 , 0, -1\}$,
  with the usual multiplication
law, and hyperaddition defined by $1 \boxplus  1 = 1 ,$\ $-1
\boxplus  -1 = -1  ,$\ $ x \boxplus  0 = 0 \boxplus  x = x $ for all
$x ,$ and $1 \boxplus  -1 = -1 \boxplus  1 = \{ 0, 1,-1\} $.

\item  The \textbf{phase hyperfield}. Let $S^1$ denote the complex
  unit circle, and take $\mcH = S^1\cup \{ 0 \}$.
  Nonzero points $a$ and $b$ are \textbf{antipodes} if $a = -b.$
Multiplication is defined as usual (so corresponds on $S^1$ to
addition of angles). We denote  an open arc of less than 180 degrees
connecting two distinct points $a,b$ of the unit circle by
$(a\,b)$. The hypersum is given, for $a,b\neq 0$, by
$$a \boxplus b=
\begin{cases} (a\,b) \text{ if } a \ne b \text{ and } a\neq -b;\\
  \{ -a,0,a \} \text{
if } a = -b , \\   \{  a \} \text{ if }   a = b.
\end{cases}$$

\item   The \textbf{tropical hyperfield}
  consists of the set $\mcH=\R \cup \{-\infty\}$, with $-\infty$
  as the zero element and $0$ as the unit, equipped
  with the addition $a\boxplus b = \{a\}$ if $a>b$,
  $a\boxplus b = \{b\}$ if $a<b$,
  and $a\boxplus a= [-\infty,a]$.

  The \textbf{signed tropical hyperfield}  $\mcH$
  is the union of two disjoint copies of $\R$,
  the first one being identified with $\R$, and denoted by $\R$,
  the second one being denoted by $(-) \R$,
  with a zero element  $-\infty$  adjoined. The signed tropical hyperfield is described in  \cite[Example~5.20]{AGR1}, and is otherwise known as the ``real tropical numbers,''
  or the ``tropical real hyperfield,'' see \cite{Vir}.
  \end{enumerate}
  \end{example}

\section{Krasner's residue construction}

We are finally equipped to return to the question, ``What can be said when we take the residues  with respect to a subgroup?''
The answer is intriguing.
The pertinence of hyperfields is Krasner's lovely construction   \cite{krasner}, known in the literature as the \textbf{quotient hyperfield}, done here more generally.

\subsection{Quotient  hypermagmas}$ $

\begin{definition}\label{kras1}
    Suppose  $\mcM$ is    a  $\tT$-bimagma, which normalizes a   submonoid  $G$ of $\tT$. Define the \textbf{quotient hypermagma} $\mcH =\mcM/G$
    to have the \textbf{hyperoperation} $\boxdot : \mcH \times \mcH \to \mathcal{P}(\mcH) $ given by $$b_1 G \boxdot b_2 G = \{ cG: c\in b_1G * b_2G\}.$$
\end{definition}

\begin{lemma}   Suppose  $\mcM$ is    a  $\tT$-bimagma, which normalizes a   submonoid  $G$ of~$\tT$.
\begin{enumerate}\eroman
    \item
    The quotient hypermagma $\mcH =\mcM/G$ is indeed   a hypermagma, under the   action $(a_1G)\boxdot
 (a_2G) = \{aG: a\in a_1a_2 G\}$.
    \item $G$ is a normal submonoid of $\tT$, so we have a monoid $\bar \tT := \tT/G$, and  $\mcH =\mcM/G$ is a   $\bar \tT$-magma,
 under the   action $(\bar a)(bG) = (ab)G$.\end{enumerate}
\end{lemma}
\begin{proof}
    (i) Clearly $(a_1G)(a_2G) =(a_1G)\boxdot(a_2G) = (a_1 a_2)G  = a_1 (a_2G) . $

(ii) Since  $\mcM$  normalizes $G$, its subset $\tT$ normalizes $G,$ so $G$ is a normal submonoid of $\tT$ and $\bar \tT$ becomes a monoid under the action $(\overline{a_1})(\overline{a_2}) = (a_1 a_2 ) G$. The action matches that of (i).
\end{proof}
\begin{example}  (Well-known special cases)
   \begin{enumerate}
       \item    Suppose  $\mcM$ is    a  $\tT$-module, which normalizes a   submonoid  $G$ of $\tT$. Define the \textbf{quotient hypermodule} $\mcH =\mcM/G$ over $\tT/G$
    to have   \textbf{hyperaddition} $\boxplus : \mcH \times \mcH \to \mathcal{P}(\mcH) $ by $$b_1 G \boxplus b_2 G = \{ cG: c\in b_1G + b_2G\}.$$

     \item  Suppose  $F$ is    a field with a subgroup $G$.  Define the \textbf{quotient hyperfield} $\mcH =F/G$   to have   \textbf{hyperaddition} as in (i) and the usual residue multiplication $(b_1G)(b_2G) = (b_1b_2)G$ for $b_i\in F.$
   \end{enumerate}
\end{example}

\begin{remark}\label{hp2} Suppose that  the hypernegative $-\one$ exists in  the quotient hypermodule $\mcH =\mcM/G$. Define $e = \one \boxplus (-\one) \in \nsets(\mcH)$ (so $0\in e$).

\begin{enumerate}\eroman
\item   $e =  \{ g_1-g_2: g_i = G\},$ and  $ee = \{ (g_1-g_2)g_3 \boxplus(g_4-g_5)g_6 :  g_i = G\}.$
     \item $e\boxplus e = \{ (g_1-g_2)  -(g_4-g_5) :  g_i = G\}.$

  \item Thus $e$ encapsulates the question in \cite{MasM}, ``Under what conditions is $F\setminus G$ spanned by (or, equals) $G-G$, for a multiplicative subgroup $G$ of a field~$F$?

     \item
If $G = \{ \pm 1\}$, then $\zero\in a \boxplus a$ for all $a\in \mcH.$

\end{enumerate}
\end{remark}

Distributivity can fail in the quotient hyperfield (the phase hyperfield of  Example~\ref{ex-krasner}   being an counterexample), thereby enhancing the following result.

\begin{lem}
   $ee = e\boxplus e$  in  $\mcS/G$,  for a $\tT$-semiring $\mcS$ with $-1\in G$.\end{lem}
\begin{proof}
    $ee \subseteq e\boxplus e$, by direct verification in Remark~\ref{hp2}(iii),(iv), taking $g_3=g_6=\one$.
    For the converse,    note that $(g_1-g_2)g_3 \boxplus (g_4-g_5)g_6
    = (g_1 g_3 -g_2g_3)\boxplus (g_4g_6 -g_5g_6).$
\end{proof}

 \subsection{Description of hyperfields as quotient hyperfields}$ $

At times we want to identify   different examples. Towards this end, an \textbf{isomorphism} $f:\mcH \to \mcJ$ of hypermonoids is defined as a   1:1 and onto map satisfying $f(a_1 \boxplus a_2) =f(a_1)\boxplus f(a_2)$ for all $a_i\in \mcH.$

\begin{example}$ $
\begin{enumerate}\eroman
\item  The {Krasner hyperfield}  is isomorphic to the quotient hyperfield $F/F^\times$, for any field $F$.

\item  The {hyperfield of signs}  is isomorphic to   the quotient hyperfield $K/K_{>0}$ for every linearly
ordered field $(K,\leq )$,
where $K_{>0}$ is the group of positive elements of $K$.

\item  The  {phase hyperfield}  is isomorphic to the quotient hyperfield
$\C/\R_{>0}$.

\item  The  \textbf{weak phase hyperfield} can be obtained by taking
  the quotient $F/G$, where $F=\C\{\{t^{\R}\}\}$, and $G$
  is the group of (generalized) Puiseux series with  positive real
  leading coefficient, where the leading coefficient is the coefficient
  $f_\lambda$ of the series $f=\sum_{\lambda \in \Lambda} f_\lambda t^\lambda$
  such that $\lambda$ is the minimal element of
  $\{\lambda \in \Lambda : f_\lambda \neq 0\}$.

  Another variant, \cite[Example~5.21]{AGR1},
  can be obtained by taking
  the quotient $K/G$, where $K=\C\{\{t^{\R}\}\}$, and $G$
  is the group of (generalized) Puiseux series with  positive real
  leading coefficient, where the leading coefficient is the coefficient
  $f_\lambda$ of the series $f=\sum_{\lambda \in \Lambda} f_\lambda t^\lambda$
  such that $\lambda$ is   minimal such that
  $ f_\lambda \neq 0$.

  \item  The {tropical hyperfield}
  is isomorphic to the quotient hyperfield $F/G$,
  where $F$ denotes a field with a surjective non-archimedean
  valuation $v: F \to \R\cup\{+\infty\}$, and
  $G:=\{f\in F: v(f) =0\}$, the equivalence
  class of any element $f$ having value $a$ being
  identified with $-a$.

  \item The signed tropical hyperfield is isomorphic to  the quotient hyperfield $K/G$
  where $K$ is a linearly ordered non-archimedean field with a surjective
  valuation $\val: K \to \R\cup\{+\infty\}$, and
  $$G:=\{f\in K: f>0, \; \val f =0\}.$$

\end{enumerate}

\end{example}

\begin{rem}
    Krasner's original use  \cite{krasner1,krasner2} of quotient hyperfields was in valuation theory, in which he succeeded in approximating complete fields of characteristic $p>0$ by fields of characteristic $0$. This aspect is developed further in~\cite{KuLS}.
\end{rem}

   \subsection{Other Krasner-type examples}$ $

   One can take hypermultiplication instead of hyperaddition (but foregoing the $\tT$-action).

 \begin{example}\label{mh}$ $\begin{enumerate}\eroman
     \item
     (as in \cite[\S 2.2]{NakR})
    Suppose  $\mcM$ is    a multiplicative  magma and $\equiv$ is an equivalence, such that the setwise product of equivalence classes is the union of equivalence classes. Then $\mcM/\equiv$ is a hypermagma, under the definition $[b_1] \boxdot [b_2] = \{ [b] \in \mcM: b \in [b_1][b_2]\}.$ If $1\in \mcM$ satisfies $[1][b] = [b]=[b][1]$ for all $b\in \mcM,$ then $[1]$ is the unit element of $\mcH.$ If furthermore $b\in \mcM$ is invertible then $[b]\in \mcM$ is invertible with hyperinverse $[b]^{-1} = [b^{-1}]$, provided $b'\equiv b$ implies $b'$ is invertible with $(b')^{-1}\equiv b^{-1}.$

       \item
    As an instance of (i),  as in \cite[\S 2.2]{NakR}, let $(\mcM,\cdot)$ be  a multiplicative monoid and $G$ is a multiplicative submonoid of $ \mcM$. \begin{enumerate}
        \item The \textbf{right coset hypermonoid} $\mcM/G$
    has multiplication given by $bG \boxdot b'G = \{cG: c \in b G \, b'G\}$ for $b,b'\in \mcM$.  This is a left magma over any subset of $\mcM$, in particular over any transversal.

  \item The \textbf{double coset hypermonoid} $\mcM//G$
    has multiplication given by  $GbG \boxdot Gb'G = \{GcG: c \in G b G b'G\}$ for $b,b'\in \mcM$.
    \end{enumerate}
 \end{enumerate}
 \end{example}

\subsubsection{m-hypersemirings and nonassociative algebras}$ $

One attains interesting constructions by reversing the role of addition and multiplication. The idea is that in any pre-semiring $\mcS$ with commutative addition, we factor out a submonoid of $(\mcS,+)$.

 \begin{definition}\label{kras3}$ $
     \begin{enumerate}\eroman
           \item An \textbf{m-hypermonoid} is a hypermonoid
$(\mathcal{H},\boxdot,\one)$, where $\one \in \mathcal{H}$ is the unit element, and
the operation
$\boxdot :\mcH \times \mcH\to\nsets(\mcH)$ is  associative in the sense that   $(a_1
\boxdot a_2) \boxdot a_3 = a_1 \boxdot (a_2\boxdot a_3), \ \forall
a_i\in \mathcal{H}.$

 \item  An \textbf{m-hypersemiring} (resp.~ \textbf{m-hypersemifield}) is an    m-hypermonoid (resp.~hypergroup) in which  $(\mathcal H,+)  $ also is an additive monoid. (Thus the  addition extended elementwise to $\nsets (\mathcal H  )$ makes $\nsets (\mathcal H  )$ a pre-semiring.)
     \end{enumerate}
 \end{definition}

\begin{example}\label{kras2}$ $\begin{enumerate}\eroman
       \item $R$  is a ring and $L$ is a left ideal. Then $R/L$ is an m-hyperring, by defining addition as in cosets, and hypermultiplication $$(r_1 +L)\boxdot (r_2 +L) = \{ (r_1 +a)r_2 +L : a \in L.\}$$
\item  In Definition~\ref{kras3}(ii) we did not require the operation $\boxdot$ to be associative, or even distributive.
Suppose $(G,0)$ is an additive subgroup of a pre-semiring $\mcS$. Define hypermultiplication $$(r_1 +G)\boxdot (r_2 +G) = \{ (r_1r_2 +r_1g_1 +g_2r_2) +G : g_1,g_2 \in G.\}$$
This can be done, for example, for $\mcS$ a
Lie semialgebra \cite{GaR2}, or a brace (or even a ``semi-brace,'' defined in the obvious way).
\end{enumerate}
\end{example}

\begin{rem} One can obtain intriguing results specializing hyperstructures to classical structures, when one has a canonical way of choosing an element. For instance, let us tie in Example~\ref{kras2}(i) with a nonassociative algebra construction of Pumpluen~\cite{Pu} which dates back at least to Dickson. Suppose $R$ is a graded algebra, such as a skew polynomial algebra, which has a division algorithm for $f\in R$, in the sense that given any $g\in R$ we can write $g=qf +r$ where $\deg r <\deg f,$ with $q,r$ uniquely determined by $g.$  We   define a nonassociative multiplication on $R/Rf$ by taking $[r_1][r_2] = [r]$ where  $r_1r_2 = qf +r$ as above.

    In the m-hyperring setting, taking $L=Rf,$ we have $(r_1 +L)\boxdot (r_2 +L) = \{ (r_1 +hf)r_2 +L : h \in R\}  = \{ (r_1r_2  +hf r_2 +L) : h \in R\}, $ and $r$ is the element of lowest in this set, so
    the map sending $[g]$ to the lowest degree element of $L+g$ identifies the residue m-hyperring   $R/L$ with Pumpluen's nonassociative algebra $R/L$.

    In other words, \cite{Pu} could be viewed as a special case of Example~\ref{kras2}(i), in the case that $R$ has a division algorithm.
\end{rem}

\section{Categorical properties of the residue construction}

The residue construction is surprisingly tractable under the Noether structure theory. First we look at bimagma homomorphisms via universal algebra.

\begin{definition}$ $
   \begin{enumerate}\eroman
   \item A  \textbf{homomorphism}  of magmas
       is a map $f:\mcM \to \mcJ$ satisfying $f(b_1*b_2) = f(b_1) * f(b_2)$ for all $b_i\in \mcM$.

      \item A  \textbf{surjection}  of  magmas
       is an onto homomorphism.
   \end{enumerate}
\end{definition}

\begin{rem} If $f:\mcM \to \mcN$ is a surjection, with $\mcM$ a weakly admissible $\tT$-bimagma, then  we can view $\mcN$ as a   bimagma over $f(\tT)$ in the obvious way.
\end{rem}

\subsubsection{ $\subseteq$-morphisms}$ $

 Morphisms are more sophisticated at the level of hypermagmas.

\begin{definition}\label{su2} \eroman
Let $(\mcH)$ (resp.~$(\mcJ,*)$) be a hypermagma.
 A $\subseteq$-\textbf{morphism},  (analogous to ``colax morphism'' in \cite{NakR}) is a map $f:\mcA \to \mcA'$ satisfying the following condition:
\begin{itemize}
        \item $f(a_1*a_2)\subseteq f(a_1)*f(a_2),\quad \forall a_1 , a_2\in \mcH.$
\end{itemize}
    \end{definition}

The following is a version of the Noether isomorphism theorems.

\begin{theorem}$ $ Suppose that  $ \mcM $ is a  $\tT$-bimagma,  and  $\mcM$ normalizes a submonoid $G$ of $\tT.$ \begin{enumerate}\eroman
    \item If  $f: \mcM  \to \mcN$ is a surjection of  magmas then   there is a $\subseteq$-morphism $\bar f :\mcM/G \to \mcN/f(G)$ given by $\bar f (bG) = \{ b'G : f(b)f(G) = f(b')f(G)\}.$

      \item If $\mcM$ normalizes  submonoids $G \subseteq G_1$ of $\tT,$ then $$\mcM/G_1 \cong (\mcM/G)/(G_1/G).$$
\end{enumerate}
\end{theorem}
 \begin{proof}
 (i)
\begin{equation}\begin{aligned}
    \bar f&(b_1G b_2G) = \{ b'G : f(b_1b_2)f(G) = f(b')f(G)\}\\& = \{ b'G : f(b_1)f(b_2)f(G) = f(b')f(G)\}\\& \subseteq  \{ b'G : f(b_1)f(G)=f(b')f(G)\} \{ b''G : f(b_2)f(G) = f(b'')f(G)\}\\& = \bar f(b_1G)\bar f(b_2G) .
\end{aligned}
\end{equation}

 (ii) The cosets of $G$ match up with the cosets of $G_1,$
 i.e., $(aG )\bar G_1$ with $aG_1.$\end{proof}

 \section{Pairs}\label{pr}$ $

 The process of modding out a subgroup has some categorical drawbacks, largely because $|\nsets(\mcH)| = 2^{|\nsets(H)|}.$ In particular $|\nsets(\mcH_1) \times \nsets(\mcH_2)| =  2^{|\nsets(H_1)|+|\nsets(H_2)|}$  whereas $|\nsets(\mcH_1 \times \mcH_2)| = 2^{|\nsets(H_1)||\nsets(H_2)|} $  . Thus,   $\nsets(\mcH_1) \times \nsets(\mcH_2) \subset \nsets(\mcH_1  \times  \mcH_2), $  and the Krasner residue construction does not respect direct products. A similar argument shows that the Krasner residue construction cannot be functorial with respect to tensor products unless one has considerable collapsing (as is the case in~\cite{NakR}). So one is led to a broader view of $\tT$-bimagmas   which includes the Krasner residue construction.

\begin{lem}\label{hp22}$ $
\begin{enumerate}\eroman
 \item     Any hypermagma $(\mathcal{H},*,\zero)$ gives rise to a $\mcH$-bimagma  $\nsets(\mcH),$ with  the  operation $(*)$  given by $$S_1 * S_2 = \cup \{s_1 * s_2: s_i\in S_i\}.$$ $(*)$   is associative on $\nsets(\mcH)$ if $\mcH$ is a hypermonoid.

    \item   When $\mathcal H$ is a $\tT$-hypermagma, $\nsets (\mathcal H  )$ is provided with a  natural action which  makes $ \nsets (\mathcal H  )$ a $\tT$-bimagma.

   \item  When $\mathcal H$ is a $\tT$-hyperring, $\nsets (\mathcal H  )$ is provided with a  natural elementwise multiplication which  makes $ \nsets (\mathcal H  )$ an  $\mcH$-pre-semiring\footnote{In general $\nsets (\mathcal H  )$ is not distributive, but does satisfy $S(S_1*S_2)\subseteq S S_1  \, *\, SS_2 $ and $(S_1*S_2)S\preceq  S_1 S \, *\, S_2S $ for all $S,S_i \in \nsets (\mathcal H  ).$  cf.~\cite[Proposition 1.1]{Mas}.}.
\end{enumerate}

\end{lem}
 \begin{proof} (i) Associative is checked elementwise.

(ii) Define   the actions $aS = \{as :  \ s\in S\}$   $Sa = \{sa : \ s\in S\}$ for $ a\in \tT, \ S\in \nsets (\mcH)$. The axioms of Definition~\ref{Tmagm}(i) are verified elementwise.

     (iii) Multiplication in $\nsets (\mathcal H  )$ is given by $S_1 \boxdot S_2 = \{ a_1\cdot a_2 : a_i \in S_i\}. $
 \end{proof}

 The main challenge is to find a replacement for $\zero$, which  plays a secondary role in residue structures to $\{ b : \zero \in bG\}$.

 \begin{definition}\label{symsyst}
We follow \cite{AGR2,JMR}, generalizing  \cite{Row16}.  \eroman
 \begin{enumerate}
     \item
A  $\tT$-\textbf{pair} $(\mcA,\mcA_0)$  (just called \textbf{pair} if $\tT$ is understood) is  a    $\tT$-bimagma $(\mcA,*,\iota)$ together with a subset $\mcA_0$ containing $\iota$, which satisfy the  property that $ab_0 , b_0 a \in \mcA_0$ for all $a\in \tT,$ $b_0\in \mcA_0.$

\item
A pair $(\mcA,\mcA_0)$ is said to be \textbf{weakly admissible} if $ \mcA$ is weakly admissible as a $ \tT$-bimagma.

\item
A  \textbf{pre-semiring pair} (resp.~\textbf{semiring pair}) $(\mcA,\mcA_0) $ is a weakly admissible pair for which $\mcA$ is a pre-semiring (resp.~semiring).

\end{enumerate}
\end{definition}

  \subsection{Surpassing relations}\label{su3}$ $

 \begin{definition}\label{sur1}$ $
   A \textbf{surpassing relation} on a weakly admissible pair $(\mathcal A,\mcA_0)$, denoted
  $\preceq$, is a pre-order satisfying the following:
\begin{enumerate}
\item
$a_1 \preceq a_2 $ for $a_1 ,a_2 \in   \tTz$ implies $a_1 =a_2.$
\item
$b  \preceq b* c$ for all $b \in \mcA$ and $c\in \mcA_0$.
\item
$b  \preceq \iota$ for $b \in \mcA$ implies $b=\iota$.
\item
$\mcA_0 \subseteq \mathcal A_{\operatorname{Null}},$ where $\mcA _{\operatorname{Null}} : = \{ c\in \mcA: \iota \preceq c\}. $
\end{enumerate}
 \end{definition}

\begin{example}\label{precex}$ $
  There are two main examples of surpassing relations.

 \begin{enumerate}\eroman  \item \label{precex-i}(Compare with \cite[Definition~2.8]{AGG2})
 Define $\preceq_{0}$ by $b_1 \preceq_{0} b_2 $ when
    $b_1 = b_2 *c$ for some $c\in \mcA_0.$ This is the kind used mainly in supertropical mathematics \cite{IR}, in \cite{Row16}, and in \cite{AGR2}.

\item On a hypermagma pair $(\mcH,\mcH_0)$ we define $\preceq_\subseteq$ on $\mcA = \nsets(\mcH)$ by set inclusion. This surpassing relation, in which $\mcA_0$ does not appear, fits in well with residue hypermagmas, and is  the main application in this note.
 \end{enumerate}\end{example}

A pre-semiring  pair $(\mathcal A,\mcA_0)$ is $\preceq$-\textbf{distributive} if $b(b_1*b_2)\preceq b b_1  \, *\, b_2b_2 $ and $(b_1*b_2)b\preceq  b_1 b \, *\, b_2b $ for all $b,b_i \in \mcA.$

\begin{rem}\label{hp23} Let $\mcH$ be a $\tT$-hypermagma.  $(\nsets(\mcH),\nsets(\mcH)_0)$ is a   pair, as noted in  Lemma~\ref{hp22}(ii), with surpassing relation $\preceq_\subseteq $.
 \end{rem}

 Likewise we have
\begin{lem}
    In Example~\ref{kras2}(i), hypermultiplication is $\preceq_\subseteq $-{distributive} over addition in $R/L.$
\end{lem}
\begin{proof}
       \begin{equation}
           \begin{aligned}(r_1 + L)&\boxdot((r_2 + L )+ (r_3 + L)) = (r_1 + L)\boxdot(r_2 + r_3+ L) \\& = \{ (r_1+a)(r_2+r_3)  +L: a\in L\}  \\& \preceq_\subseteq \{ (r_1+a)r_2+ L  : a\in L\}+  \{ (r_1+a')r_3+ L  : a'\in L\}
           \\& = (r_1 + L)\boxdot(r_2 + L )+ (r_1 + L)\boxdot(r_3 + L).
           \end{aligned}
       \end{equation}
\end{proof}

\subsection{Property N and negation maps}\label{propN1}

\begin{definition}\label{propN}$ $ \begin{enumerate}
    \item  A  pair  $(\mcA,\mcA_0)$ satisfies
 \textbf{Property~N} when  \begin{enumerate}\eroman
      \item  For each $a\in \tT$ there is $a^\dag \in \tT$ (not necessarily uniquely defined) such that  $a * a^\dag \in \mcA_0.$ We define $ a^\circ  : = a * a^\dag .$ 

      \item  $ a * a' = a^\circ $ for each $a,a'\in \tT$ such that  $a*a' \in \mcA_0$. (Thus $a^\circ$ is uniquely defined.)
 \end{enumerate}

   \item  When $(\mcA,+) $ is an admissible $\tT$-monoid and $(\one^\dag)^\dag = \one $, the map $a\mapsto (-)a : = a \one^\dag$ is   called a \textbf{negation map}. The pair $(\mathcal A,\mcA_0)$  is \textbf{uniquely negated} if $a+b\in \mcA_0$ implies $b = (-)a,$ for $a\in \tT.$
\end{enumerate}
 \end{definition}

 \subsection{Examples of pairs}$ $

\begin{example}\label{prex}$ $
 \begin{enumerate}\eroman
          \item    This note largely concerns
 $\mcA = \nsets (\mcH)$ where $\mcH$ is a hypersemiring. Let $\mcA_0 =\{S \in \nsets (\mcH): 0 \in S\}$ and $\tT = \mcH\setminus\{\zero\}.$ Then  $(\mcA,\mcA_0)$ is an admissible  $\preceq_\subseteq $-{distributive} pre-semiring  pair,  in view of the footnote to~Lemma~\ref{hp22}(ii). (More generally, we could take $\mcH$ to be a $\tT$-hypersemiring where $\tT$ is a given monoid.)

 A {\it weak neutral element} of a hypermagma is defined in \cite{NakR}  to be an element $\bar \iota \in \mcH$ such that $a \in (a *\bar \iota)\cap (\bar \iota *a)$ for all $a\in \mcH.$ Thus $\mcA_0$ can be identified with the weak neutral elements of $\mcH$, cf.~\cite[Remark~3.14(v)]{JuR2}.

 When $\mcH$ is a hyperring,
the hypernegation makes $(\mcA,\mcA_0)$  uniquely negated.

\item \textbf{Supertropical pairs}, based on \cite{IR}, are described in \cite[Example~4.2]{AGR2}, and are fundamental to tropical algebra. They satisfy $a_1+a_2 \in \mcA_0$ if and only if $a_1=a_2 ,$ so $(\mcA,\mcA_0)$ is uniquely negated.

\item Any $\tT$-bimagma $\mcA$ defines a pair, by taking $\mcA_0 = \{\zero\}.$

 \item There is a doubling procedure given in \cite[\S 4.4, esp.~Lemma 4.20]{AGR2} which embeds a pair into a pair with a negation map, and preserves inclusion for hyperpairs.

  \item (\cite[Example 4.1]{AGR2}) For  an arbitrary  monoid $\tT$, take $\mcA = \tTz \cup \{\infty\},$ where $a_1+a_2= \infty$  and $a \infty = \infty = a+\infty = \infty +a$ for all $a, a_i$ in $\tT$. $\mcA_0 =  \{0,\infty\}.$
$(\mcA,\mcA_0)$ has the negation map $a \mapsto a,$ but is not uniquely negated. This pair plays a key role in \cite{NakR}. If instead we declared $a+a = a$ then we still would have Property N, but without a negation map.
 \end{enumerate}
 \end{example}

Other examples of pairs are given in
\cite[\S 4]{AGR2} and \cite[\S 3.1]{JuR2}.


\section{Categorical aspects of pairs}

There is a weaker version of morphism for hypermagmas.
\begin{definition}
    \label{mo1}$ $ \begin{enumerate}
        \item  A \textbf{weak morphism}\footnote{This is the standard definition for hypergroups, but could be vacuous for hypermagmas lacking hypernegatives.}  $f:\mcH \to \mcJ$ of hypermagmas satisfies $\iota\in   a_1 * a_2$ implies $\iota \in f(a_1)*f(a_2)$ for $a_i\in \mcH.$
    \item  A \textbf{weak morphism} $f:(\mcA,\mcA_0) \to (\mcA',\mcA'_0)$  of pairs satisfies $    a_1 * a_2 \in \mcA_0 $ implies $  f(a_1)*f(a_2) \in \mcA_0'$ for all $a_i\in \mcH.$
    \end{enumerate}
\end{definition}

Better, we can insert the surpassing relation into our categories.

\begin{definition}\label{su21}
Let $(\mcA,*,\iota)$ (resp.~$(\mcA',*,\iota')$) be  weakly admissible bimagmas over a monoid $\tT$ (resp.~$\tT'$).
  \begin{enumerate}\eroman
      \item
When $\mcA'$ has a surpassing relation $\preceq$,  a $\preceq$-\textbf{map} is a   map $f:\mcA \to \mcA'$ satisfying the following conditions:
\begin{itemize}
\item $ f(\iota) = \iota'.$
\item $
f(\tT) \subseteq  \tT' .  $
  \end{itemize}
\item In this case, a $\preceq$-\textbf{morphism}  $f:(\mcA,\mcA_0) \to (\mcA',\mcA'_0)$  satisfies the  following conditions:
\begin{itemize}
    \item $f(b_1) \preceq  f(b_2),\quad \forall b_1 \preceq b_2\in \mcA.$
       \item $f(ab) = f(a)f(b)$, \ $f(ba) = f(b)f(a)$, for $a\in \tT,$ $b\in \mcA.$
       \item $f(b_1*b_2)\preceq f(b_1)*f(b_2),\quad \forall b_1 , b_2\in \mcA.$
\end{itemize}

\end{enumerate}
\end{definition}


\begin{lem}\label{wkm}(As in {\cite[Lemma~2.10]{AGR1}})
    Every  $\preceq$-{morphism} of a pair is a weak morphism.
\end{lem}

 \begin{lem}\label{ext1}
     Any $\subseteq$-morphism $f: \mcH \to \mcJ$ of $\tT$-bimagmas extends to a    $\preceq$-morphism $\hat f: \nsets(\mcH) \to  \nsets(\mcJ)$ given by $\hat f (S) = \{f(a):  a\in S\},$ with $\hat f ( \nsets(\mcH))_0 \subseteq  \nsets(\mcJ)_0 .$
 \end{lem}
 \begin{proof}
     Clearly if $S_1 \subseteq S_2$ then $f(S_1)\subseteq f(S_2),$ and \begin{equation}
         \begin{aligned}
             f(S_1*S_2) & = \{ f(a): a \in S_1 *S_2\} = \{f (a_1*a_2): a_i \in S_i\}\\& \subseteq  \{f (a_1)*f(a_2): a_i \in S_i\} =  f(S_1)*f(S_2)
         \end{aligned}
     \end{equation}
     If $\iota\in S$ then $\iota = f(\iota) \in f(S).$
 \end{proof}
\subsection{Categorical constructions}
\begin{definition}
We list the basic categories.
  \begin{enumerate}\eroman
  \item (from \cite{NakR})\begin{enumerate}
        \item  {\bf uHMag} is the category of  hypermagmas, with $\subseteq$-morphisms.
     \item  {\bf cMsc}   is the category of commutative hypermonoids, with $\subseteq$-morphisms.
    \end{enumerate}

    \item   $\tT$-{\bf uHMag} is the category of  $\tT$-hypermagmas, with $\subseteq$-morphisms.
          \item   m-{\bf HMod} is the category of  m-hypermodules, with $\subseteq$-morphisms.
    \item   $\tT$-{\bf HMod} is the category of  $\tT$-hypermodules (where $\tT$ can vary, and the morphism acts like a monoid homomorphism on $\tT$), with $\subseteq$-morphisms.
   \item $\tT$-{\bf  Pr} is the category of $\tT$-pairs  with   $\preceq$-morphisms.
    \item m-{\bf Pr} is the category of m-pairs (where $\tT$ can vary, and the morphism acts like a monoid homomorphism on $\tT$)  with   $\preceq$-morphisms.

    \item $\tT$-{\bf ModPr} is the category of $\tT$-module pairs  with   $\preceq$-morphisms.
    \item  {\bf semiringPr} is the category of semiring pairs  with   $\preceq$-morphisms.
  \end{enumerate}

 For any category {\bf C}, the name {\bf w-C},   means the same objects, but the morphisms are weak morphisms.

\end{definition}

One reason to consider pairs is that they are preserved under the main constructions of category theory.

\begin{rem}\label{cons}$ $
\begin{enumerate}\eroman

     \item The direct product of monoids is a monoid.  Hence, the direct product of $\tT_i$-pairs $(\mcA_i,{\mcA_i}_0)$ is the  $\prod \tT_i$-pair $\prod(\mcA_i,{\mcA_i}_0)$, viewed componentwise. Thus,  m-{\bf Pr} is closed under products, as is  {\bf semiringPr}.

     \item The direct sum of $\tT $-pairs $(\mcA_i,{\mcA_i}_0)$ is the  $ \tT $-pair $\oplus (\mcA_i,{\mcA_i}_0)$, under the diagonal action $a(b_i) = (ab_i).$ Hence, $\tT$-{\bf  Pr} is closed under sums,
     as is its subcategory  $\tT$-{\bf ModPr}.

     \item If $(\mcA,\mcA_0)$ is a pre-semiring pair, then
     \begin{enumerate}
         \item for a  commuting associative indeterminate  $\lambda$,    $(\mcA[\lambda],\mcA_0[\lambda])$ is a pair
     over the monoid of monomials $\cup_{i\ge 0} \tT  \lambda^i,$  where multiplication is given by $$(\sum b_i \lambda^i)(\sum b'_j \lambda^j) = \sum _k (\sum_{i+j=k} b_ib_j') \lambda ^k.$$
 \item    the matrix pair $(M_n(\mcA),M_n(\mcA_0))$ exists over   $\cup _{i,j} \tT e_{i,j}\cup \{\zero, \one\}, $
where $e_{i,j}$ are the usual matrix units.     \end{enumerate} Hence, {\bf semiringPr} is closed under polynomial extensions and matrix ring extensions.
     \item Tensor products of $\tT$-pairs exist, following the lines of \cite{CHWW}, but taking the $\tT$-action into account. Namely,
     one defines the \textbf{free $\tT$-magma} $(\mcF(X),*)$ on a set $X$.

\begin{definition} Let   $\mathcal M_1$
 be a right $\tT$-magma
 and $\mathcal M_2$ be a left  $\tT$-magma, and $\mathcal F_i = \mcF(M_i), i=1,2$.
\begin{enumerate}\eroman
    \item  Define the \textbf{$\tT$-tensor product magma} $\mathcal M_1  \otimes _{\tT} \mathcal M_2$ to be the magma $(\mathcal F_1 \oplus \mathcal F_2)/\Cong,$
where $\Cong$ is the
congruence
 generated by all \begin{equation}\label{defcong}\bigg(\big( v_1*w_1,x_2\big),  \big( v_1, x_{2})*(w_1,  x_{2})\big)\bigg) ,\quad \bigg(\big( x_{1},   v_2*w_2\big),  \big( (x_{1},v_2 ) *(x_{1},w_2 ) \big)\bigg) ,\end{equation}
\begin{equation}
    \label{defcong1} \bigg((  x_1 a, x_2 ), (x_1,a
 x_2
),\bigg)
\end{equation}
  $ \forall x_{i}\in \mathcal M_i, $ $ v_i,w_i \in \mcF_i,\, a \in
\tT$.    \end{enumerate}
 \end{definition}
 If $(\mcM_i,{\mcM_i}_0)$ are pairs, then we define $$(\mathcal M_1  \otimes _{\tT} \mathcal M_2)_0 = \mathcal M_1  \otimes _{\tT} {\mathcal M_2}_0 + {\mathcal M_1}_0  \otimes _{\tT} \mathcal M_2,$$ to obtain a pair.
There are some rather subtle issues, which are treated in~\cite{JuR2}.
     \end{enumerate}
\end{rem}

\subsection{Connections between hypermagma categories and categories of pairs}

\begin{theorem}$ $
    \begin{enumerate}\eroman
\item  There is a fully faithful functor from $\tT$-{\bf uHMag} to $\tT$-{\bf Pr}, by passing to the power set, which restricts to  a  faithful    functor from   $\tT$-{\bf HMod} to $\tT$-{\bf ModPr}.
  \item The Krasner residue map induces a  faithful    functor from   $\tT$-{\bf Pr}  to  $\tT$-{\bf uHMag}, which restricts to  a  faithful    functor from $\tT$-{\bf ModPr}  to  $\tT$-{\bf HMod}.

  \item The tensor product is a faithful functor from $\tT$-{\bf ModPr} to $\tT$-{\bf ModPr}.

    \item For each category {\bf C} listed above, the identity map induces faithful functor from {\bf w-C} to {\bf C}.

\end{enumerate}
\end{theorem}
 \begin{proof} In all of these, one checks  that the morphisms match.

     (i)  By Example~\ref{prex}(i).

     (ii) The Krasner residue of a  $\tT$-pair is a $\tT$-pair, by Lemma~\ref{hp22}.

   (iii)   The tensor product of morphisms  is a morphism, by \cite[Proposition~4.4]{JuR2}.

       (iv)  By Lemma~\ref{wkm}.

 \end{proof}

\bibliographystyle{amsalpha}

\end{document}